\documentclass[reqno,12pt]{amsart}
\usepackage{amsmath,amssymb,color}
\usepackage{amsthm}
\usepackage{comment}

\newtheorem{theorem}{Theorem}

\newtheorem{proposition}{Proposition}

\def\re{\mathbb{R}}

\def\N{\mathbb{N}}

\def\eps{\varepsilon}

\def\pd{\partial}
\def\ol{\overline}
\def\ul{\underline}

\def\la{\lambda}

\def\({\left(}
\def\){\right)}

\def\pd{\partial}

\def\|{\Vert}

\begin{document}
\title[Asymptotic behavior of radial solutions]{Notes on asymptotic behavior of radial solutions for some weighted elliptic equations on the annulus}

\author{Futoshi Takahashi}

\address{
Department of Mathematics, Osaka Metropolitan University \\
3-3-138, Sumiyoshi-ku, Sugimoto-cho, Osaka, Japan \\
}

\email{futoshi@omu.ac.jp}

\begin{abstract}
In this paper, we study the asymptotic behavior of radial solutions for several weighted elliptic equations with power type or exponential type nonlinearities on an annulus. 
%as the nonlinear exponent $p \to +\infty$ or $p \to 1$.
\end{abstract}

%34C23: Bifurcation
%37G99: Local and nonlocal bifurcation theory (none of the above)
% 26D10: Inequalities involving derivatives and differential and integral operators
% 26D15: Inequalities for sums, series and integrals
% 35A23: Inequalities involving derivatives and differential and integral operators, inequalities for integrals
% 46E35: Sobolev spaces and other spaces of ``smooth" functions, embedding theorems, trace theorems

\date{\today}

\dedicatory{}

\maketitle

\section{Introduction.}

Let $A = \{ x \in \re^N \ | \ 0 < a < |x| < b \}$ be an annulus in $\re^N$, $N \ge 2$.
In this paper, we study the following elliptic problem with the inverse-square weight in $\re^2$:
\begin{equation}
\label{Henon(N=2)}
	\begin{cases}
	-\Delta u = \dfrac{u^p(x)}{|x|^2}, \quad &x \in A \subset \re^2, \\
	u(x) > 0, \quad &x \in A, \\
	u(x) = 0, \quad &x \in \pd A, 
	\end{cases}
\end{equation}
where $p > 1$.
In higher dimensions $N \ge 3$, we also consider
\begin{equation}
\label{Henon(N>2)}
	\begin{cases}
	-\Delta u = (N-2)^2 \dfrac{u^p(x)}{|x|^{2(N-1)}}, \quad &x \in A \subset \re^N, \\
	u(x) > 0, \quad &x \in A, \\
	u(x) = 0, \quad &x \in \pd A. 
	\end{cases}
\end{equation}
Since the embedding of the radial Sobolev space $H^1_{0, rad}(A) \hookrightarrow L^{p+1}(A)$ is compact for any $p > 1$, 
we can easily obtain a radially symmetric solution to \eqref{Henon(N=2)} or \eqref{Henon(N>2)}, by considering a suitable minimization problem. 
Later it will be shown that the positive radial solution is unique for both problems and it is denoted by $u_p$.
In this paper, we study the asymptotic behavior of the radial solution $u_p$ to \eqref{Henon(N=2)} or \eqref{Henon(N>2)} 
as $p \to \infty$ and $p \searrow 1$.
In the following, we denote
\begin{equation}
\label{Def:L_p}
	L_p = \int_0^1 \frac{ds}{\sqrt{1-s^{p+1}}}, \quad (p>0).
\end{equation}

\begin{theorem}
\label{Thm:N=2(global)}
For any $p > 1$, the problem \eqref{Henon(N=2)} admits a unique radially symmetric solution $u_p(x) = u_p(|x|)$.
Furthermore, it holds that
\begin{equation}
\label{p_large(N=2)} 
	u_p(|x|) \to \frac{2}{\log b - \log a} 
	\begin{cases}
	\log |x| - \log a, &\quad \text{for} \ a \le |x| \le \sqrt{ab} \\
	\log b - \log |x|, &\quad \text{for} \ \sqrt{ab} \le |x| \le b
	\end{cases} 
\end{equation}
in $C^0(\ol{A})$ as $p \to \infty$, and
\begin{equation}
\label{p_1(N=2)}
	u_p(|x|) = \(\frac{2}{\log b - \log a}\)^{\frac{2}{p-1}} \(\frac{p+1}{2}\)^{\frac{1}{p-1}} L_p^{\frac{2}{p-1}} \left[ \sin \( \frac{\pi(\log b - \log |x|)}{\log b - \log a} \) + o_p(1) \right]
\end{equation}
in $C^0(\ol{A})$ as $p \searrow 1$, where $o_p(1) \to 0$ in $C^0(\ol{A})$ as $p \searrow 1$.
\end{theorem}

\begin{theorem}
\label{Thm:N>2(global)}
For any $p > 1$, the problem \eqref{Henon(N>2)} admits a unique radially symmetric solution $u_p(x) = u_p(|x|)$.
Furthermore, it holds that
\begin{equation}
\label{p_large(N>2)} 
	u_p(|x|) \to \frac{2}{a^{2-N} - b^{2-N}} 
	\begin{cases}
	a^{2-N} - |x|^{2-N}, &\quad \text{for} \ a \le |x| \le \(\frac{a^{2-N} + b^{2-N}}{2}\)^{\frac{1}{2-N}} \\
	|x|^{2-N} - b^{2-N}, &\quad \text{for} \ \(\frac{a^{2-N} + b^{2-N}}{2}\)^{\frac{1}{2-N}} \le |x| \le b
	\end{cases} 
\end{equation}
in $C^0(\ol{A})$ as $p \to \infty$, and
\begin{equation}
\label{p_1(N>2)} 
	u_p(|x|) = \(\frac{2}{a^{2-N} - b^{2-N}}\)^{\frac{2}{p-1}} \(\frac{p+1}{2}\)^{\frac{1}{p-1}} L_p^{\frac{2}{p-1}} \left[ \sin \( \frac{\pi(|x|^{2-N} - b^{2-N})}{a^{2-N} - b^{2-N}} \) + o_p(1) \right]
\end{equation}
in $C^0(\ol{A})$ as $p \searrow 1$, where $o_p(1) \to 0$ in $C^0(\ol{A})$ as $p \searrow 1$.
\end{theorem}

Note that in both Theorem \ref{Thm:N=2(global)} and Theorem \ref{Thm:N>2(global)},
$L^{\infty}$-blow up of $u_p$ does not occur when $p \to \infty$.
On the other hand, $\| u_p \|_{L^{\infty}(A)}$ could diverge as $p \searrow 1$ if the annulus is ``thin" in the sense that
$\log b - \log a < 2$ when $N=2$, or $a^{2-N} - b^{2-N} < 2$ when $N \ge 3$.

Also there are local convergence results as follows.

\begin{theorem}
\label{Thm:N=2(local)}
Let $u_p(x) = u_p(r)$, $r = |x|$, be the unique radial solution to \eqref{Henon(N=2)} for $p > 1$.
Define $\eps_p > 0$ and $\tilde{u}_p$ by
\begin{equation}
\label{utilde(N=2)}
	\begin{cases}
	&p \eps_p^2 \| u_p \|_{L^{\infty}(A)}^{p-1} \equiv 1, \\
	&\tilde{u_p}(t) = \frac{p}{\| u_p \|_{L^{\infty}(A)}} \left\{ u_p \(e^{-(\eps_p t - \frac{\log ab}{2})}\) - \| u_p \|_{L^{\infty}(A)} \right \}, \\
	&t \in I_p = (-\frac{\log b - \log a}{\eps_p}, \frac{\log b - \log a}{\eps_p}). 
	\end{cases}
\end{equation}
Then $\eps_p \to 0$ and
\[
	\tilde{u_p}(t) \to U(t) = \log \frac{4 e^{\sqrt{2}t}}{(1+e^{\sqrt{2}t})^2} \quad \text{in} \ C^1_{loc}(\re)
\]
as $p \to \infty$.
\end{theorem}

\begin{theorem}
\label{Thm:N>2(local)}
Let $u_p(x) = u_p(r)$, $r = |x|$, be the unique radial solution to \eqref{Henon(N>2)} for $p > 1$.
Define $\eps_p > 0$ and $\tilde{u}_p$ by
\begin{align}
\label{utilde(N>2)}
	\begin{cases}
	&p \eps_p^2 \| u_p \|_{L^{\infty}(A)}^{p-1} \equiv 1, \\
	&\tilde{u_p}(t) = \frac{p}{\| u_p \|_{L^{\infty}(A)}} \left\{ u_p \(\(\eps_p t - \frac{a^{2-N} + b^{2-N}}{2}\)^{\frac{1}{2-N}}\) - \| u_p \|_{L^{\infty}(A)} \right \}, \\
	&t \in I_p = (-\frac{a^{2-N}-b^{2-N}}{\eps_p}, \frac{a^{2-N} - b^{2-N}}{\eps_p}).
	\end{cases}
\end{align}
Then $\eps_p \to 0$ and
\[
	\tilde{u_p}(t) \to U(t) = \log \frac{4 e^{\sqrt{2}t}}{(1+e^{\sqrt{2}t})^2} \quad \text{in} \ C^1_{loc}(\re)
\]
as $p \to \infty$.
\end{theorem}

Note that in Theorem \ref{Thm:N=2(local)} and Theorem \ref{Thm:N>2(local)}, 
$U(t)$ is the unique solution of the 1-D Liouville equation
\begin{equation}
\label{Liouville(whole)}
	\begin{cases}
	-U''(t) = e^{U(t)}, \quad t \in \re, \\
	U(0) = U'(0) = 0
	\end{cases}
\end{equation}
with $\int_{\re} e^{U(t)} dt = 2\sqrt{2}$.

In \cite{Grossi(JDE)}, M. Grossi considered the asymptotic behavior as $p \to \infty$ of radial solutions $u_p$ of the problem
\[
	\begin{cases}
	-\Delta u = u^p, \quad &\text{in} \ A \subset \re^N, \\
	u > 0, \quad &\text{in} \ A, \\
	u = 0, \quad &\text{on} \ \pd A 
	\end{cases}
\]
on an annulus $A = \{ a < |x| < b \}$ in $\re^N$, $N \ge 2$.
He proved that the same asymptotic formulae \eqref{p_large(N=2)}, \eqref{p_large(N>2)} hold true for $u_p$ according to the case $N=2$ or $N \ge 3$.  
Note that the limit functions in \eqref{p_large(N=2)}, \eqref{p_large(N>2)} are the constant multiple of $G_{a,b}(|x|, r_0)$,
where $G_{a,b}(r,s)$ is the Green's function of the operator $-\frac{d^2}{dr^2} - \frac{N-1}{r} \frac{d}{dr}$ on $r \in (a, b)$
defined by
\begin{align}
\label{Green_ab}
	&G_{a,b}(r,s) = \frac{s^{N-1}}{(N-2)(a^{2-N}-b^{2-N})} \begin{cases} 
			(b^{2-N}-s^{2-N})(r^{2-N}-a^{2-N}), \quad \text{for} \ a \le r \le s, \\
			(s^{2-N}-b^{2-N})(b^{2-N}-r^{2-N}), \quad \text{for} \ s < r \le b,
			\end{cases} (N \ge 3), \notag \\
	&G_{a,b}(r,s) = \frac{s}{\log b - \log a} \begin{cases} 
			(\log r - \log a)(\log b - \log s), \quad \text{for} \ a \le r \le s, \\
			(\log s - \log a)(\log b - \log r), \quad \text{for} \ s < r \le b, \\
			\end{cases} (N = 2),
\end{align}
and $r_0 \in (a, b)$ is defined by
\begin{equation}
\label{r_0}
	r_0 = \begin{cases}
	&\( \frac{a^{2-N}+b^{2-N}}{2}\)^{\frac{1}{2-N}}, \quad (N \ge 3), \\
	&\sqrt{ab}, \quad (N = 2).
	\end{cases}
\end{equation}
Furthermore he proved the local convergence result that 
\[
	\tilde{u_p}(t) \to U(t) = \log \frac{4 e^{\sqrt{2}t}}{(1+e^{\sqrt{2}t})^2} \quad \text{in} \ C^1_{loc}(\re), 
\]
where $\eps_p > 0$ and $\tilde{u_p}$ are defined as
\[
	\begin{cases}
	&p \eps_p^2 \| u_p \|_{L^{\infty}(A)}^{p-1} \equiv 1, \\
	&\tilde{u_p}(t) = \frac{p}{\| u_p \|_{L^{\infty}(A)}} \left\{ u_p(\eps_p t + r_p ) - \| u_p \|_{L^{\infty}(A)} \right\}, \\
	&u_p(r_p) = \| u_p \|_{L^{\infty}(A)}, \quad r_p \in (a,b), \\
	&t \in (\frac{a-r_p}{\eps_p}, \frac{b-r_p}{\eps_p}).
	\end{cases}
\]

\vspace{1em}
Theorems \ref{Thm:N=2(global)}-\ref{Thm:N>2(local)} are strongly related with results in \cite{Grossi(JDE)} and the proofs are much simpler.
The key ingredients are the explicit norm computations and the asymptotic behaviors of the unique solution $W_p$ of 1-D Emden equation \eqref{Eq:W_p} as $p \to \infty$ or $p \searrow 1$.
Based on these, we prove theorems via the well-known transformations.

\section{Asymptotic behaviors of $W_p$.}

For any $p > 1$ and $\bar{a}, \bar{b} \in \re$ with $\bar{a} < \bar{b}$, let $W_p$ denote the unique function satisfying
\begin{equation}
\label{Eq:W_p}
	\begin{cases}
	-W_p''(s) = W_p^p(s), \quad s \in I = (\bar{a}, \bar{b}), \\
	W_p(s) > 0, \quad s \in I, \\
	W_p(\bar{a}) = W_p(\bar{b}) = 0.
	\end{cases}
\end{equation}
By \cite{GNN}, $W_p$ is symmetric with respect to $s = s_0 = \frac{\bar{a}+\bar{b}}{2}$ and
$\| W_p \|_{L^{\infty}(\bar{a}, \bar{b})} = W_p(\frac{\bar{a}+\bar{b}}{2})$.
Also $W_p$ is strictly monotone increasing on $(\bar{a}, s_0)$ and strictly monotone decreasing on $(s_0, \bar{b})$.

\vspace{1em}
The following explicit norm computations are done by T. Shibata \cite{Shibata(JMAA)}. 
For the readers convenience, we give the proof here.

\begin{proposition}(\cite{Shibata(JMAA)})
\label{Prop:W_p(norm)}
For a given $p > 1$, let $W_p$ be the unique solution of the problem \eqref{Eq:W_p}.
Then
\begin{align}
\label{xi_p}
	&\xi_p = \| W_p \|_{L^{\infty}(\bar{a}, \bar{b})} = \(\frac{2}{\bar{b}-\bar{a}} \sqrt{\frac{p+1}{2}} L_p \)^{\frac{2}{p-1}}, \\ 
\label{L^qW_p}
	&\| W_p \|_{L^q(\bar{a}, \bar{b})}^q = \sqrt{\frac{2}{p+1}} \xi_p^{\frac{2q-p+1}{2}} B\(\frac{q+1}{p+1},\frac{1}{2}\), \quad (q > 0)
\end{align}
where $L_p$ is defined in \eqref{Def:L_p} and $B(x,y) = \int_0^1 t^{x-1}(1-t)^{y-1} dt$ denotes the Beta function.
\end{proposition}
Especially by \eqref{L^qW_p} with $q = p$, we see
\begin{equation}
\label{L^pW_p}
	\| W_p \|_{L^p(\bar{a}, \bar{b})}^p = \(\frac{2}{\bar{b}-\bar{a}} L_p \)^{\frac{p+1}{p-1}} \(\frac{p+1}{2}\)^{\frac{1}{p-1}} B(1, \frac{1}{2}) \to  \frac{2\pi}{\bar{b}-\bar{a}} 
\end{equation}
as $p \to \infty$.

\begin{proof}
We use a time-map method to compute $\| W_p \|_{L^q(\bar{a}, \bar{b})}$.
 
Multiply the equation by $W_p'(s)$, we have
\begin{align*}
	\left\{-\frac{1}{2}(W_p'(s))^{2}-\frac{1}{p+1}W_p^{p+1}(s)\right\}' = 0,
\end{align*}
thus 
\begin{align*}
	-\frac{1}{2}(W_p'(s))^{2}-\frac{1}{p+1}W_p^{p+1}(s) = \text{constant} = 0-\frac{1}{p+1} \xi_p^{p+1}.
\end{align*}
Therefore, we have
\begin{equation}
\label{Wprime}
	W_p'(s) = \sqrt{\frac{2}{p+1}(\xi_p^{p+1} - W_p^{p+1}(s))}.
\end{equation}
By \eqref{Wprime} and the change of variables $t = W_p(s)$, we compute
\begin{align*}
	\frac{\bar{b}-\bar{a}}{2} = \int_{\bar{a}}^{\frac{\bar{a}+\bar{b}}{2}} 1 ds 
	&=\int_{\bar{a}}^{\frac{\bar{a}+\bar{b}}{2}} \frac{W_p'(s)ds}{\sqrt{\frac{2}{p+1}(\xi_p^{p+1} - W_p^{p+1}(s))}} \\
	&=\sqrt{\frac{p+1}{2}}\int_{0}^{\xi_p}\frac{dt}{\sqrt{\xi_p^{p+1} - t^{p+1}}} \\
	&=\sqrt{\frac{p+1}{2}}\xi_p^{\frac{1-p}{2}} \int_0^1 \frac{d\tau}{\sqrt{1 - \tau^{p+1}}}.
\end{align*}
From this, we obtain
\[
	\xi_p = \| W_p \|_{L^{\infty}(\bar{a}, \bar{b})} = \( \frac{2}{\bar{b}-\bar{a}} \sqrt{\frac{p+1}{2}} L_p \)^{\frac{2}{p-1}}.
\]
Next, by the formula $\xi_p$ and the change of variables $t = W_p(s)$ again, we see
\begin{align*}
	\| W_p \|_{L^q(\bar{a}, \bar{b})}^q &= 2\int_{\bar{a}}^{\frac{\bar{a}+\bar{b}}{2}} |W_p(s)|^q ds \\
	&=2\int_{\bar{a}}^{\frac{\bar{a}+\bar{b}}{2}} W_p^{q}(s)\frac{W_p'(s)ds}{\sqrt{\frac{2}{p+1}(\xi_p^{p+1} - W_p^{p+1}(s))}}\\
	&=\sqrt{2(p+1)}\int_{0}^{\xi_p}\frac{t^q dt}{\sqrt{\xi_p^{p+1} - t^{p+1}}} \\
	&=\sqrt{2(p+1)}\xi_p^{\frac{2q-p+1}{2}}\int_{0}^{1}\frac{\tau^q dt}{\sqrt{1-\tau^{p+1}}} \\
	&=\sqrt{\frac{2}{p+1}} \xi_p^{\frac{2q-p+1}{2}} B\(\frac{q+1}{p+1},\frac{1}{2}\),
\end{align*}
which proves the formula of $\| W_p \|_{L^q(\bar{a}, \bar{b})}$.
\end{proof}

The following two propositions were essentially proven by M. Grossi \cite{Grossi(JDE)}; see Theorem 1.1, and Proposition 3.2 in \cite{Grossi(JDE)}. 
Though Theorem 1.1 (1.3) in \cite{Grossi(JDE)}, $N \ge 3$ is assumed but the proof there works for the case $N=1$.
Also when $N=1$, $\bar{a} > 0$ is not needed.
We give proofs here for the sake of the convenience of the readers.

\begin{proposition}
\label{Prop:W_p(global)}
Let $W_p$ be the unique solution of the problem \eqref{Eq:W_p}.
Then
\begin{align*}
	W_p(s) \to &\(\frac{4}{\bar{b}-\bar{a}}\) G\(s,\frac{\bar{a}+\bar{b}}{2}\)
	= \frac{2}{\bar{b}-\bar{a}} 
	\begin{cases} 
	(s-\bar{a}) \quad \bar{a} \le s \le \frac{\bar{a}+\bar{b}}{2} \\
	(\bar{b}-s) \quad \frac{\bar{a}+\bar{b}}{2} \le s \le \bar{b}
	\end{cases} \quad \text{in} \ C^0[\bar{a}, \bar{b}]
\end{align*}
as $p \to \infty$ holds true.
Here $G(s,t)$ is the Green's function of the operator $-\frac{d^2}{ds^2}$ with the Dirichlet boundary condition on $[\bar{a}, \bar{b}]$:
\begin{equation}
\label{Green}
	G(s,t) = \frac{1}{\bar{b}-\bar{a}} 
	\begin{cases}
		(s-\bar{a})(\bar{b}-t), \quad \bar{a} \le s \le t \le \bar{b}, \\
		(\bar{b}-s)(t-\bar{a}), \quad \bar{a} \le t \le s \le \bar{b}.
	\end{cases}
\end{equation}
\end{proposition}

\begin{proof}
Put $I = (\bar{a}, \bar{b})$.
For $\bar{a} < r < s_0 = \frac{\bar{a} + \bar{b}}{2}$, we have from \eqref{Eq:W_p} that
\[
	\int_{r}^{s_0} -W_p''(s) ds = \int_{r}^{s_0} W_p^p(s) ds \le \int_{\bar{a}}^{\bar{b}} W_p^p(s) ds = \| W_p \|_{L^p(I)}^p.
\]
Thus we have $W_p'(r) \le \| W_p \|_{L^p(I)}^p$, $(\bar{a} < \forall r < s_0)$.
Similarly, we see $-W_p'(r) \le \| W_p \|_{L^p(I)}^p$, $(s_0 < \forall r < \bar{b})$,
so we obtain
\begin{equation}
\label{AA1}
	|W_p'(s)| \le \| W_p \|_{L^p(I)}^p = O(1) \quad (\bar{a} < \forall s < \bar{b})
\end{equation}
as $p \to \infty$ by \eqref{L^pW_p}.
Then by Ascoli-Arzela Theorem, there exists a function $W: I \to \re$ such that 
\[
	W_p \to W \quad \text{uniformly on} \  \bar{I} = [\bar{a}, \bar{b}] 
\]
as $p \to \infty$ (up to a subsequence).
Since $\| W_p \|_{L^{\infty}(I)} = \xi_p \to 1$ as $p \to \infty$ by Proposition \ref{Prop:W_p(norm)}, we assure that $\| W \|_{L^{\infty}(I)} = W(s_0) = 1$.

Now, we claim that for any $s \in [\bar{a}, \bar{b}]$, $s \ne s_0 = \frac{\bar{a}+\bar{b}}{2}$, there exists $p_0 > 1$ such that
\begin{equation}
\label{claim1}
	W_p(s) < 1 \quad (p > p_0)
\end{equation}
holds true.
We prove the claim for $s < s_0$ first.
Indeed, assume the contrary that there exists $\bar{s} < s_0$ and $p_n \to \infty$ such that $W_{p_n}(\bar{s}) \ge 1$ for any $n \in \N$.
Since $W_{p_n}$ is strictly monotone increasing on $(\bar{s}, s_0)$, we obtain $W_{p_n}(t) > W_{p_n}(\bar{s}) \ge 1$ for any $t \in (\bar{s}, s_0)$.
Then there exists $\eps_0 > 0$ such that $W_{p_n}(t) > 1 + \eps_0$ for any $t \in (\bar{s}, s_0)$, 
which implies that
\[
	W_{p_n}'(\bar{s}) = \int_{\bar{s}}^{s_0} -W_{p_n}''(s) ds = \int_{\bar{s}}^{s_0} W_{p_n}^{p_n}(s) ds \ge (1+\eps_0)^{p_n} (s_0 - \bar{s}) \to +\infty
\]
as $p_n \to \infty$. This contradicts to \eqref{AA1}.
The proof when $s > s_0$ is similar. Therefore, we obtain \eqref{claim1}.

Note that by the equation of $W_p$ and \eqref{claim1}, the convergence $W_p \to W$ actually holds in $C^1([\bar{a}, \bar{b}] \setminus \{ s_0\})$.
Thus letting $p \to \infty$ with \eqref{claim1} in \eqref{Eq:W_p}, we obtain that the limit function satisfies that
\[
	\begin{cases}
	-W''(s) = 0, \quad s \in I = (\bar{a}, \bar{b}) \setminus \{ s_0 \}, \\
	W(\bar{a}) = W(\bar{b}) = 0, \\
	W(s_0) = 1.
	\end{cases}
\]
From this, we have the conclusion that
\[
	W(s) = \frac{2}{\bar{b}-\bar{a}} 
	\begin{cases}
		(s-\bar{a}), \quad (\bar{a} \le s < s_0) \\
		(\bar{b}-s), \quad (s_0 < s \le \bar{b})
	\end{cases}
	= \(\frac{4}{\bar{b}-\bar{a}}\) G(s,s_0),
\]
where $G(s,t)$ is the Green's function of the operator $-\frac{d^2}{ds^2}$ with the Dirichlet boundary condition on $[\bar{a}, \bar{b}]$ defined in \eqref{Green}.
\end{proof}

\begin{proposition}
\label{Prop:W_p(local)}
Let $W_p$ be the unique solution of the problem \eqref{Eq:W_p}.
Define $\eps_p > 0$ and $\tilde{W}_p$ be such that 
\begin{equation}
\label{Wtilde}
	\begin{cases}
	&p \eps_p^2 \| W_p \|_{L^{\infty}(\bar{a}, \bar{b})}^{p-1} \equiv 1, \\
	&\tilde{W}_p(t) = \frac{p}{\| W_p \|_{L^{\infty}(\bar{a}, \bar{b})}} \left\{ W_p(\eps_p t + s_0) - W_p(s_0) \right\}, \\
	&t \in I_p = (\frac{\bar{a}-s_0}{\eps_p}, \frac{\bar{b}-s_0}{\eps_p}). 
	\end{cases}
\end{equation}
Then $\eps_p \to 0$ and
\[
	\tilde{W}_p(t) \to U(t) = \log \frac{4 e^{\sqrt{2}t}}{(1+e^{\sqrt{2}t})^2} \quad \text{in} \ C^1_{loc}(\re)
\]
as $p \to \infty$.
\end{proposition}

\begin{proof}
By Proposition \ref{Prop:W_p(norm)} \eqref{xi_p} and $L_p > 1$ for any $p > 1$, we observe
\[
	\| W_p \|_{L^{\infty}(\bar{a}, \bar{b})}^{p-1} = \xi_p^{p-1} = \(\frac{2}{\bar{b}-\bar{a}} \sqrt{\frac{p+1}{2}} L_p \)^2 > \(\frac{2}{\bar{b}-\bar{a}}\)^2 \( \frac{p+1}{2} \).
\]
Thus $\xi_p \ge C p^{\frac{1}{p-1}}$ and $\eps_p = \frac{1}{\sqrt{p \xi_p^{p-1}}} \le C' (\frac{1}{p})$ for some $C, C' > 0$. 
Then $\eps_p \to 0$ and ``$I_p \to \re$" as $p \to \infty$.
Direct computation shows that $\tilde{W}_p$ solves the equation
\begin{equation}
\label{Eq:Wtilde}
	\begin{cases}
	&-\tilde{W}_p''(t) = \( 1 + \frac{\tilde{W}_p(t)}{p} \)^p, \\ 
	&t \in I_p = (\frac{\bar{a}-s_0}{\eps_p}, \frac{\bar{b}-s_0}{\eps_p}) =: (\bar{a}_p, \bar{b}_p), \\ 
	&\tilde{W}_p(\bar{a}_p) = W(\bar{b}_p) = 0, \\
	&\tilde{W}_p(0) = 0, \tilde{W}'_p(0) = 0. 
	\end{cases}
\end{equation}
Also by \eqref{AA1} and the definition of $\eps_p$, we see
\begin{align*}
	|\tilde{W}_p'(t)| &= \frac{p}{\| W_p \|_{L^{\infty}(\bar{a}, \bar{b})}} \eps_p |W_p'(\eps_p t + s_0)| \\
	&= \frac{p^{\frac{1}{2}}}{\| W_p \|^{\frac{p+1}{2}}_{L^{\infty}(\bar{a}, \bar{b})}} |W_p'(s)| \Big|_{s = \eps_p t + s_0} \\
	&\le  C \frac{p^{\frac{1}{2}}}{(p^{\frac{1}{p-1}})^{\frac{p+1}{2}}} = C p^{-\frac{1}{p-1}} = O(1)
\end{align*}
for any $t \in I_p$ as $p \to \infty$.
Thus by Ascoli-Arzela Theorem, we assure that there exists a function $U: \re \to \re$ such that
\[
	\tilde{W}_p \to U \quad \text{uniformly on compact sets on} \ \re.
\]
By using the equation of $\tilde{W}_p$, we obtain the uniform boundedness of $\tilde{W}_p''$ in $p$ so this convergence holds actually on $C^1_{loc}(\re)$.
Letting $p \to \infty$ in \eqref{Eq:Wtilde}, we see that the limit function satisfies
\[
	\begin{cases}
	-U''(t) = e^{U(t)}, \quad t \in \re, \\
	U(0) = U'(0) = 0.
	\end{cases}
\]
By the uniqueness of the initial value problem of the second order ODE, we conclude that 
\[
	U(t) = \log \frac{4 e^{\sqrt{2}t}}{(1+e^{\sqrt{2}t})^2}.
\]
\end{proof}

Next proposition concerns the behavior of $W_p$ as $p \searrow 1$ and is proven in \cite{CSLin}.
In \cite{CSLin}, the asymptotic behavior as $p \searrow 1$ of solutions of
\[
	\begin{cases}
	-\Delta u = u^p \quad \text{in} \ \Omega \subset \re^N, \\
	u > 0 \quad \text{in} \ \Omega \subset \re^N, \\
	u = 0 \quad \text{on} \ \pd\Omega,
	\end{cases}
\]
is considered, where $\Omega$ is a bounded convex domain. 
Here we give a simple proof of the fact in one-dimensional case.

\begin{proposition}
\label{Prop:W_p(pto1)}
Let $W_p$ be the unique solution of the problem \eqref{Eq:W_p}.
Then
\begin{align*}
	W_p(s) &= \| W_p \|_{L^{\infty}(\bar{a}, \bar{b})} \left[ \sin \( \frac{\pi(s - \bar{a})}{\bar{b}-\bar{a}} \) + o_p(1) \right] \\
	&= \(\frac{2}{\bar{b}-\bar{a}}\)^{\frac{2}{p-1}} \(\frac{p+1}{2}\)^{\frac{1}{p-1}} L_p^{\frac{2}{p-1}} 
\left[ \sin \( \frac{\pi(s - \bar{a})}{\bar{b}-\bar{a}} \) + o_p(1) \right]
\end{align*}
in $C^1([\bar{a}, \bar{b}])$ as $p \searrow 1$, where $o_p(1) \to 0$ in $C^1([\bar{a}, \bar{b}])$ as $p \searrow 1$.
\end{proposition}

\begin{proof}
Let $w_p = \frac{W_p}{\| W_p \|_{L^{\infty}(\bar{a}, \bar{b})}} = \frac{W_p}{\xi_p}$. 
Then $\| w_p \|_{L^{\infty}(\bar{a}, \bar{b})} = 1$ and $w_p$ solves
\begin{equation}
\label{Eq:w_p}
	\begin{cases}
	-w_p''(s) = \xi_p^{p-1} w_p^p(s), \quad s \in I = (\bar{a}, \bar{b}), \\
	w_p(s) > 0, \quad s \in I, \\
	w_p(\bar{a}) = w_p(\bar{b}) = 0.
	\end{cases}
\end{equation}
By Proposition \ref{Prop:W_p(norm)} \eqref{xi_p}, we see
\[
	\xi_p^{p-1} = \(\frac{2}{\bar{b}-\bar{a}} \)^2 L_p^2 \( \frac{p+1}{2} \) \to \( \frac{\pi}{\bar{b}-\bar{a}} \)^2
\]
as $ p \searrow 1$. Note that $L_1 = \int_0^1 \frac{ds}{\sqrt{1-s^2}} = \frac{\pi}{2}$.
Thus $w_p''$ is uniformly bounded on $[\bar{a}, \bar{b}]$ and the Ascoli-Arzela Theorem implies that there exists a function $w: [\bar{a}, \bar{b}] \to \re$ such that
$w_p \to w$ as $p \searrow 1$ in $C^1([\bar{a}, \bar{b}])$.
Taking a limit $p \searrow 1$ in \eqref{Eq:w_p}, we see $w$ satisfies 
\[
	\begin{cases}
	-w''(s) = \( \frac{\pi}{\bar{b}-\bar{a}} \)^2 w(s), \quad s \in I = (\bar{a}, \bar{b}), \\
	w(s) > 0, \quad s \in I, \\
	w(\bar{a}) = w(\bar{b}) = 0.
	\end{cases}
\]
Thus $w$ is the first eigenfunction of $-\frac{d^2}{ds^2}$ with the Dirichlet boundary condition on $I = (\bar{a}, \bar{b})$,
which leads to
\[
	w(s) = C \sin \(\frac{\pi(s-\bar{a})}{\bar{b}-\bar{a}}\), \quad \text{in} \ C^1([\bar{a}, \bar{b}])
\]
for some $C > 0$.
Since $\| w \|_{L^{\infty}(I)} = C = \lim_{p \searrow 1} \| w_p \|_{L^{\infty}(I)} = 1$, 
we conclude
\[
	\lim_{p \to 1} w_p = \lim_{p \to 1} \frac{W_p(s)}{\xi_p} = \sin \(\frac{\pi(s-\bar{a})}{\bar{b}-\bar{a}}\)
\]
in $C^1([\bar{a}, \bar{b}])$, which ends the proof.
\end{proof}

\section{Proofs of Theorems.}

Theorem \ref{Thm:N=2(global)}-\ref{Thm:N>2(local)} in \S 1 will be proven by combining well-known transformations 
and the asymptotic behaviors of $W_p$ as $p \to \infty$, or $p \searrow 1$ in \S 2.
First we prove Theorem \ref{Thm:N=2(global)}.

\begin{proof}
For a radially symmetric solution $u_p(|x|) = u_p(r)$ to \eqref{Henon(N=2)},
define a new function $v_p$ as 
\begin{equation}
\label{transform(N=2)}
	u_p(r) = v_p(s), \quad s = -\log r, \quad s \in (-\log b, -\log a).
\end{equation}
Then a direct computation shows that $v_p$ solves
\[
	\begin{cases}
	-v_p''(s) = v_p^p(s), \quad s \in (-\log b, -\log a), \\
	v_p(s) > 0, \quad s \in (-\log b, -\log a), \\
	v_p(-\log b) = v_p(-\log a) = 0.
	\end{cases}
\]
By the uniqueness, $v_p \equiv W_p$, where $W_p$ is the unique solution of \eqref{Eq:W_p} with $\bar{a} = -\log b$, $\bar{b} = -\log a$. 
Thus the radial solution $u_p$ to \eqref{Henon(N=2)} is unique for any $p > 1$ and is represented as
\begin{equation}
\label{u_pform(N=2)}
	u_p(r) = W_p(s), \quad s = -\log r, \quad s \in (-\log b, -\log a).
\end{equation}
Applying Proposition \ref{Prop:W_p(global)} to $W_p$ and going back to $u_p$ with $s = -\log |x|$, 
we obtain \eqref{p_large(N=2)}. 
On the other hand, applying Proposition \ref{Prop:W_p(pto1)} to $W_p$ and going back to $u_p$ with $s = -\log |x|$ and \eqref{xi_p}, 
we obtain \eqref{p_1(N=2)}. 
\end{proof}

Proof of Theorem \ref{Thm:N>2(global)} is similar:

\begin{proof}
For a radially symmetric solution $u_p(|x|) = u_p(r)$ to \eqref{Henon(N>2)},
define a new function $v_p$ as 
\begin{equation}
\label{transform(N>2)}
	u_p(r) = v_p(s), \quad s = r^{2-N}, \quad s \in (b^{2-N}, a^{2-N}).
\end{equation}
Then a direct computation and the uniqueness of $W_p$ assures that
\begin{equation}
\label{u_pform(N>2)}
	u_p(r) = W_p(s), \quad s = r^{2-N}, \quad s \in (b^{2-N}, a^{2-N}),
\end{equation}
where $W_p$ is the unique solution of \eqref{Eq:W_p} with $\bar{a} = b^{2-N}$, $\bar{b} = a^{2-N}$. 
Applying Proposition \ref{Prop:W_p(global)} to $W_p$ and going back to $u_p$ with $s = |x|^{2-N}$, 
we obtain \eqref{p_large(N>2)}. 
Applying Proposition \ref{Prop:W_p(pto1)} to $W_p$ with \eqref{xi_p}, 
we obtain \eqref{p_1(N>2)}. 
\end{proof}

Proofs of Theorem \ref{Thm:N=2(local)} and Theorem \ref{Thm:N>2(local)} is as follows:

\begin{proof}
By \eqref{u_pform(N=2)} and Proposition \ref{Prop:W_p(local)} with $\bar{a} = -\log b$, $\bar{b} = -\log a$, 
we obtain the conclusion of Theorem \ref{Thm:N=2(local)}.
Note that $\| u_p \|_{L^{\infty}(A)} = \| W_p \|_{L^{\infty}(\bar{a}, \bar{b})}$
and $s = \eps_p t + s_0 = -\log r$,
so \eqref{Wtilde} can be rewritten as \eqref{utilde(N=2)}.

For the proof of Theorem \ref{Thm:N>2(local)}, 
just we use \eqref{u_pform(N>2)} and Proposition \ref{Prop:W_p(local)} with $\bar{a} = b^{2-N}$, $\bar{b} = a^{2-N}$. 
In this case \eqref{Wtilde} is the same as \eqref{utilde(N>2)}.
\end{proof}

\section{Other application.}

Let $N \ge 1$ and consider the following elliptic equation on an annulus:
\begin{equation}
\label{Hardy-Henon}
	\begin{cases}
	-\Delta u - \frac{C_N}{|x|^2} u = |x|^{\frac{(N-1)(p-1)}{2}} u^p, \quad x \in A = \{ a < |x| < b \} \subset \re^N, \\
	u > 0, \quad x \in A, \\
	u = 0 \quad x \in \pd A,
	\end{cases}
\end{equation}
where $C_N = \frac{(N-1)(N-3)}{4}$.

\begin{theorem}
\label{Thm:Hardy-Henon}
For any $p > 1$, the positive radially symmetric solution $u_p(|x|) = u_p(r)$ of \eqref{Hardy-Henon} exists and unique.
Moreover, the following asymptotic behaviors hold:
\begin{enumerate}
\item[(i)]
As $p \to \infty$, it holds
\begin{align*}
	u_p(|x|) \to &\(\frac{4}{b-a}\) |x|^{\frac{N-1}{2}} G\(|x|,\frac{a+b}{2}\)
	= \frac{2}{b-a} |x|^{\frac{N-1}{2}}
	\begin{cases} 
	(|x|-a), \quad (a \le |x| \le \frac{a+b}{2}) \\
	(b-|x|), \quad (\frac{a+b}{2} \le |x| \le b)
	\end{cases} 
\end{align*}
in $C^0(\ol{A})$.
\item[(ii)]
Put $\eps_p > 0$ and $\tilde{u}_p$ such that 
\[
	\begin{cases}
	&p \eps_p^2 \| |\cdot|^{-\frac{N-1}{2}} u_p \|_{L^{\infty}(A)} \equiv 1, \\
	&\tilde{u}_p(t) = \frac{p}{\| |\cdot|^{-\frac{N-1}{2}} u_p \|_{L^{\infty}(A)}} \left\{ (\eps_p t + s_0)^{-\frac{N-1}{2}} u_p(\eps_p t + s_0) - \| |\cdot|^{-\frac{N-1}{2}} u_p \|_{L^{\infty}(A)} \right\}, \\
	&t \in (-\frac{b-a}{\eps_p}, \frac{b-a}{\eps_p}). 
	\end{cases}
\]
Then $\eps_p \to 0$ as $p \to \infty$ and
\[
	\tilde{u}_p(t) \to U(t) = \log \frac{4 e^{\sqrt{2}t}}{(1+e^{\sqrt{2}t})^2} \quad \text{in} \ C^1_{loc}(\re)
\]
as $p \to \infty$.
\item[(iii)]
As $p \searrow 1$, it holds
\begin{align*}
	u_p(|x|) = |x|^{\frac{N-1}{2}} \| |\cdot|^{-\frac{N-1}{2}} u_p \|_{L^{\infty}(A)} \left[ \sin \( \frac{\pi(|x| - a)}{b-a} \) + o_p(1) \right] \\
\end{align*}
in $C^1(\ol{A})$, where $o_p(1) \to 0$ in $C^1(\ol{A})$ as $p \searrow 1$.
\end{enumerate}
\end{theorem}

\begin{proof}
Let $u_p(|x|) = u_p(r)$ be any positive radial solution to \eqref{Hardy-Henon} and put a new function $v_p$ by
\[
	v_p(r) = r^{-\frac{N-1}{2}} u_p(r).
\]
Then a direct computation shows that 
\[
	\begin{cases}
	-v_p''(r) = v_p^p(r), \quad r \in (a, b), \\
	v_p(r) > 0, \quad r \in (a, b), \\
	v_p(a) = v_p(b) = 0.
	\end{cases}
\]
Thus we have $v_p \equiv W_p$, where $W_p$ is the unique solution to \eqref{Eq:W_p} with $\bar{a} = a$, $\bar{b} = b$.
Therefore positive radial solution $u_p$ to \eqref{Hardy-Henon} is unique and can be written as $u_p(r) = r^{\frac{N-1}{2}} W_p(r)$.
Then we obtain the results by applying Proposition \ref{Prop:W_p(global)}-\ref{Prop:W_p(pto1)}.
\end{proof}

\section{Exponential nonlinearity case.}

In \cite{Gladiali-Grossi(AA)}, the authors considered the Liouville equation
\begin{equation}
\label{Eq:L}
	\begin{cases}
	-\Delta u = \la e^u, \quad &\text{in} \ A, \\
	u = 0, \quad &\text{on} \ \pd A, 
	\end{cases}
\end{equation}
where $\la > 0$ and  $A = \{ x \in \re^N \ | \ 0 < a < |x| < b \}$ is an annulus in $\re^N$, $N \ge 2$.
It is well-known that the problem \eqref{Eq:L} adimits no solution if $\la$ is sufficiently large, 
but there exist at least two radially symmetric solutions $\ul{u}_{\la}$ and $U_{\la}$ for $\la>0$ sufficiently small.
$\ul{u}_{\la}$ is called the minimal solution to \eqref{Eq:L}, which is strictly stable and bounded, 
and the other solution $U_{\la}$ is unstable and can be obtained, for example, by the Mountain Pass Theorem applied in the radial Sobolev space $H^1_{0, rad}(A)$.
The authors in \cite{Gladiali-Grossi(AA)} studied the asymptotic behavior of non-minimal (blowing-up) radially symmetric solutions $U_{\la}$ to \eqref{Eq:L} as $\la \to +0$,
and obtained the following results:
Let $U_{\la}(r) = U_{\la}(|x|)$ be any radially symmetric solution to \eqref{Eq:L} such that $\gamma_{\la} = \la \int_{A} e^{U_{\la}} dx \to +\infty$ as $\la \to +0$.
Define $\delta_{\la} > 0$ by $\la \delta_{\la}^2 e^{\| U_{\la} \|_{L^{\infty}(a,b)}} \equiv 1$ and
$\tilde{U}_{\la}(t) = U_{\la}(\delta_{\la} t + r_{\la}) - U_{\la}(r_{\la})$, $t \in (\frac{a-r_{\la}}{\delta_{\la}}, \frac{b-r_{\la}}{\delta_{\la}})$,  
where $r_{\la} \in (a,b)$ be such that $U_{\la}(r_{\la}) = \| U_{\la} \|_{L^{\infty}(a,b)}$.
Then
\[
	\tilde{U}_{\la}(t) \to U(t) = \log \frac{4 e^{\sqrt{2}t}}{(1+e^{\sqrt{2}t})^2} \quad \text{in} \ C^1_{loc}(\re),
\]
and
\[
	\delta_{\la} U_{\la}(|x|) \to 2\sqrt{2} G_{a,b}(r, r_0)
\]
in $C^0(\ol{A})$ as $\la \to +0$, where $G_{a,b}(r,s)$ is the Green's function of $-\frac{d^2}{dr^2} - \frac{N-1}{r} \frac{d}{dr}$ with the Dirichlet boundary conditions 
defined in \eqref{Green_ab} and $r_0$ is in \eqref{r_0}. 

\vspace{1em}
For 1-D case, let $\bar{a}, \bar{b} \in \re$ with $\bar{a} < \bar{b}$ and consider
\begin{equation}
\label{Eq:L(1-D)}
	\begin{cases}
	-w''(s) = \la e^{w}, \quad s \in I = (\bar{a}, \bar{b}), \\
	w(\bar{a}) = w(\bar{b}) = 0.
	\end{cases}
\end{equation}
By the maximum principle, any solution to \eqref{Eq:L(1-D)} is positive on $I$,
and by \cite{GNN}, it is symmetric with respect to $s = s_0 = \frac{\bar{a}+\bar{b}}{2}$, which is the unique maximum point of $w$ in $I$.
It is a classical fact that (see for example \cite{JL}), 
there exists $\la_* > 0$ such that \eqref{Eq:L(1-D)} adimits no solution for $\la > \la_*$,
only one solution for $\la = \la_*$, just two solutions $\ul{w}_{\la}$ and $W_{\la}$ for $0 < \la < \la_*$.
Here $\ul{w}_{\la}$ is the minimal solution, and $W_{\la}$ is the unique unstable solution 
which satisfies $\| W_{\la} \|_{L^{\infty}(I)} \to +\infty$ as $\la \to +0$.

The following is the one-dimensional counter part of the results in \cite{Gladiali-Grossi(AA)}, and the proof of it is much simpler.

\begin{proposition}
\label{Prop:W(exp)}
Let $\la >0$ sufficiently small and let $W_{\la}$ be the unique unstable solution of the problem \eqref{Eq:L(1-D)}.
Then we have
\[
	\gamma_{\la} = \la \int_{\bar{a}}^{\bar{b}} e^{W_{\la}(s)} ds \to +\infty \quad (\la \to +0).
\] 
Define $\delta_{\la} > 0$ and $\tilde{W}_{\la}$ such that 
\begin{equation}
\label{Wtilde(exp)}
	\begin{cases}
	&\la \delta_{\la}^2 e^{\| W_{\la} \|_{L^{\infty}(\bar{a},\bar{b})}} \equiv 1, \\
	&\tilde{W}_{\la}(t) =  W_{\la}(\delta_{\la} t + s_0) - W_{\la}(s_0), \quad s_0 = \frac{\bar{a} + \bar{b}}{2}, \\
	&t \in I_{\la} = (\frac{\bar{a}-s_0}{\delta_{\la}}, \frac{\bar{b}-s_0}{\delta_{\la}}).  
	\end{cases}
\end{equation}
Then $\delta_{\la} \to 0$ and
\begin{equation}
\label{W(local)}
	\tilde{W}_{\la}(t) \to U(t) = \log \frac{4 e^{\sqrt{2}t}}{(1+e^{\sqrt{2}t})^2} \quad \text{in} \ C^1_{loc}(\re)
\end{equation}
as $\la \to +0$.
Furthermore, 
\begin{equation}
\label{W(global)}
	\delta_{\la} W_{\la}(s) \to 2\sqrt{2} G\(s, s_0\)
	= \sqrt{2}
	\begin{cases} 
	(s-\bar{a}), \quad (\bar{a} \le s \le \frac{\bar{a}+\bar{b}}{2}), \\
	(\bar{b}-s), \quad (\frac{\bar{a}+\bar{b}}{2} \le s \le \bar{b}),
	\end{cases} \quad \text{in} \ C^0[\bar{a}, \bar{b}]
\end{equation}
holds as $\la \to +0$. 
Here $G(s,t)$ is the Green's function of the operator $-\frac{d^2}{ds^2}$ with the Dirichlet boundary condition on $[\bar{a}, \bar{b}]$
defined by \eqref{Green}.
%\begin{equation}
%\label{Green}
%	G(s,t) = \frac{1}{\bar{b}-\bar{a}} 
%	\begin{cases}
%		(s-\bar{a})(\bar{b}-t), \quad \bar{a} \le s \le t \le \bar{b}, \\
%		(\bar{b}-s)(t-\bar{a}), \quad \bar{a} \le t \le s \le \bar{b}.
%	\end{cases}
%\end{equation}
\end{proposition}

\begin{proof}
Integrating the equation $-W_{\la}''(s) = \la e^{W_{\la}(s)}$ on $[s_0, s]$ when $s > s_0$ and on $[s, s_0]$ when $s < s_0$,
we have as in the derivation of \eqref{AA1} that
\begin{equation}
\label{AA2}
	|W_{\la}'(s)| \le \gamma_{\la} = \la \int_{\bar{a}}^{\bar{b}} e^{W_{\la}(s)} ds, \quad \forall s \in I.
\end{equation}
Therefore, if $\gamma_{\la} = O(1)$ as $\la \to +0$, Ascoli-Arzela Theorem implies that there exists $W_0$ such that
$W_{\la} \to W_0$ uniformly on $I$ as $\la \to +0$ (up to a subsequence). 
By Green's representation formula, we see
\[
	W_{\la}(s) = \la \int_I G(s,t) e^{W_{\la}(t)} dt 
\]
for $s \in I$. Letting $\la \to +0$, we obtain $W_0(s) = 0$ for any $s \in I$.
Also it is known that the minimal solution $\ul{u}_{\la}$ to \eqref{Eq:L(1-D)} satisfies $\lim_{\la \to 0} \ul{u}_{\la} = 0$ uniformly on $I$.
This fact contradicts to the local uniqueness of the minimal solution $\ul{u}_{\la}$, 
since the minimal solution could be obtained by the implicit function theorem, which results in the local uniqueness of $\ul{u}_{\la}$.
Thus we have proved that $\gamma_{\la} \to \infty$ as $\la \to +0$.
Since $\gamma_{\la} = \la \int_{\bar{a}}^{\bar{b}} e^{W_{\la}(s)} ds \le \la (\bar{b}-\bar{a}) e^{\| W_{\la} \|_{L^{\infty}(I)}}$,
we have
\[
	\delta_{\la}^2 = \frac{1}{\la e^{\| W_{\la} \|_{L^{\infty}(I)}}} \to 0, \quad (\la \to 0).
\]
Next, integrating the equation $-W_{\la}''(s) = \la e^{W_{\la}(s)}$ on $[\bar{a}, \bar{b}]$ and noting that $W_{\la}'(\bar{a}) = -W_{\la}'(\bar{b})$,
we see $\gamma_{\la} = W_{\la}'(\bar{a}) - W_{\la}'(\bar{b}) = 2W_{\la}'(\bar{a})$.
Also, multiplying $W_{\la}'(s)$ to the equation and integrating on $[\bar{a}, s_0]$, we have
\[
	\int_{\bar{a}}^{s_0} \( -\frac{1}{2} (W_{\la}'(s))^2 \)' ds  = \la \int_{\bar{a}}^{s_0} \( e^{W_{\la}(s)} \)' ds,
\]
so
\[
	-\frac{1}{2} \left\{ 0 - (W_{\la}'(a))^2 \right\} = \la \( e^{W_{\la}(s_0)} - 1\).
\]
Multiplying $\delta_{\la}^2$ to the above and noting the definition of $\delta_{\la}$, we get
\[
	\frac{1}{2} \( \frac{1}{2} \gamma_{\la} \delta_{\la} \)^2 = 1 - \la \delta_{\la}^2 \to 1 \quad (\la \to +0).
\]
Therefore we obtain that
\[
	\lim_{\la \to +0} \gamma_{\la} \delta_{\la} = 2\sqrt{2}.
\]
Note that $\tilde{W}_{\la}$ satisfies the equation
\begin{equation}
\label{Eq:Wtilde(exp)}
	\begin{cases}
	-\tilde{W}_{\la}''(t) = e^{\tilde{W}_{\la}(t)}, \quad t \in I_{\la} 
	= (\frac{\bar{a}-s_0}{\delta_{\la}}, \frac{\bar{b}-s_0}{\delta_{\la}}), \\
	\tilde{W}_{\la}(\frac{\bar{a}-s_0}{\delta_{\la}}) = \tilde{W}_{\la}(\frac{\bar{b}-s_0}{\delta_{\la}}) = 0,\\
	\tilde{W}_{\la}(0) = \tilde{W}_{\la}'(0) = 0. 
	\end{cases}
\end{equation}
Now, by \eqref{Wtilde(exp)} and \eqref{AA2}, we see
\[
	|\tilde{W}_{\la}'(t)| = \delta_{\la} |W_{\la}'(\delta_{\la} t + s_0)| \le \delta_{\la} \gamma_{\la} = O(1)
\]
as $\la \to +0$.
Also $I_{\la} = \frac{I-s_0}{\delta_{\la}} ``\to \re"$ as $\la \to +0$.
Therefore, Ascoli-Arzela Theorem implies that there exists $U_0: \re \to \re$ such that $\tilde{W}_{\la} \to U_0$ uniformly on compact sets on $\re$.
By using \eqref{Eq:Wtilde(exp)}, this convergence holds on $C^1_{loc}(\re)$ and the limit function $U_0$ solves \eqref{Liouville(whole)}.
Thus $U_0(t) \equiv U(t) = \log \frac{4 e^{\sqrt{2}t}}{(1+e^{\sqrt{2}t})^2}$ and \eqref{W(local)} is proven.

Next, we prove \eqref{W(global)}.
By Green's representation formula, we have
\begin{align*}
	W_{\la}(s) &= \la \int_{\bar{a}}^{\bar{b}} G(s, \tau) e^{W_{\la}(\tau)} d\tau 
	= \la \int_{\frac{\bar{a}-s_0}{\delta_{\la}}}^{\frac{\bar{b}-s_0}{\delta_{\la}}} G(s, \delta_{\la} t + s_0) e^{W_{\la}(\delta_{\la} t + s_0)} \delta_{\la} dt \\
	&= \la \delta_{\la} e^{\| W_{\la} \|_{L^{\infty}(I)}} \int_{I_{\la}} G(s, \delta_{\la} t + s_0) e^{\tilde{W}_{\la}(t)} dt \\
	&= \frac{1}{\delta_{\la}} \int_{I_{\la}} G(s, \delta_{\la} t + s_0) e^{\tilde{W}_{\la}(t)} dt.
\end{align*}
%where $I_{\la} = \frac{I-s_0}{\delta_{\la}} \to \re$ as $\la \to +0$.
Thus we have
\begin{equation}
\label{delta_laU_la}
	\delta_{\la} W_{\la}(s) = \int_{I_{\la}} G(s, \delta_{\la} t + s_0) e^{\tilde{W}_{\la}(t)} dt.
\end{equation}
On the other hand, we can show the estimate
\begin{equation}
\label{key_estimate}
	\tilde{W}_{\la}(t) \le C_1 U(t) + C_2
\end{equation}
for some constants $C_1, C_2 > 0$.
The proof of this fact is completely the same as Proposition 4.1 in \cite{Gladiali-Grossi(AA)}.
Then by \eqref{W(local)} and $e^{\tilde{W}_{\la}} \le Ce^{CU(t)} \in L^1(\re)$ by \eqref{key_estimate},
we can use the Lebesgue's dominated convergence theorem in \eqref{delta_laU_la} to obtain
\[
	\lim_{\la \to +0} \delta_{\la} W_{\la}(s) = G(s, s_0) \int_{\re} e^{U(t)} dt = 2\sqrt{2} G(s, s_0).
\] 
By \eqref{AA2} and $\delta_{\la} \gamma_{\la} = O(1)$, this convergence holds uniformly on $[\bar{a}, \bar{b}]$.
Thus we have proved Proposition \ref{Prop:W(exp)}.
\end{proof}

As in the derivation of Theorem \ref{Thm:N=2(global)}-\ref{Thm:N>2(local)}, 
we obtain the following assertion from Proposition \ref{Prop:W(exp)} and the transformation \eqref{transform(N=2)}.

\begin{theorem}
\label{Thm:(L)N=2}
There exists $\la_* > 0$ such that the problem
\begin{equation}
\label{Henon(L)(N=2)}
	\begin{cases}
	-\Delta u = \la \dfrac{e^{u(x)}}{|x|^2}, \quad &x \in A = \{ a < |x| < b \} \subset \re^2, \\
	u(x) = 0, \quad &x \in \pd A, 
	\end{cases}
\end{equation}
admits no radially symmetric solution for $\la > \la_*$,
one radially symmetric solution for $\la = \la_*$, 
and just two radially symmetric solutions $\ul{u}_{\la}$ and $U_{\la}$ for $0 < \la < \la_*$.
Also it holds that
\begin{align*}
	\la \int_A \frac{e^{\ul{u}_{\la}}}{|x|^2} dx \to 0, \quad
	\la \int_A \frac{e^{U_{\la}}}{|x|^2} dx \to +\infty,
\end{align*}
and $\ul{u}_{\la} \to 0$ uniformly on $\ol{A}$ as $\la \to +0$.

Define $\delta_{\la} > 0$ and $\tilde{U}_{\la}$ by
\begin{equation}
	\begin{cases}
	&\la \delta_{\la}^2 e^{\| U_{\la} \|_{L^{\infty}(A)}} \equiv 1, \\
	&\tilde{U}_{\la}(t) = U_{\la} \(e^{-(\delta_{\la} t - \frac{\log ab}{2})}\) - \| U_{\la} \|_{L^{\infty}(A)}, \\
	&t \in \tilde{I_{\la}} = (-\frac{\log b - \log a}{2\delta_{\la}}, \frac{\log b - \log a}{2\delta_{\la}}). 
	\end{cases}
\end{equation}
then $\delta_{\la} \to 0$ as $\la \to +0$ and
\[
	\tilde{U}_{\la}(t) \to U(t) = \log \frac{4 e^{\sqrt{2}t}}{(1+e^{\sqrt{2}t})^2} \quad \text{in} \ C^1_{loc}(\re)
\]
as $\la \to +0$.
Also
\[
	\delta_{\la} U_{\la}(|x|) \to 2\sqrt{2} G(-\log |x|, -\log \sqrt{ab})
\]
as $\la \to +0$ uniformly in $C^0(\ol{A})$,
where $G(s,t)$ is the Green's function of the operator $-\frac{d^2}{ds^2}$ with the Dirichlet boundary condition on $[\bar{a}, \bar{b}]$
defined in \eqref{Green} with $\bar{a} = -\log b$, $\bar{a} = -\log a$.
\end{theorem}

\begin{proof}
For a radially symmetric solution $u_\la(|x|) = u_\la(r)$ to \eqref{Henon(L)(N=2)},
we use a transformation \eqref{transform(N=2)} to define $v_\la$, i.e.,
\[
	u_{\la}(r) = v_{\la}(s), \quad s = -\log r, \quad s \in (-\log b, -\log a).
\]
Then a direct computation shows that $v_{\la}$ solves
\[
	\begin{cases}
	-v_{\la}''(s) = \la e^{v_{\la}(s)}, \quad s \in (-\log b, -\log a), \\
	v_{\la}(-\log b) = v_{\la}(-\log a) = 0.
	\end{cases}
\]
By \cite{JL}, there exists $\la_* > 0$ such that $v_{\la}$ is either $\ul{w}_{\la}$ or $W_{\la}$ if $0 < \la < \la_*$. 
Also by the transformation above, we see
\[
	\la \int_{-\log b}^{-\log a} e^{v_{\la}(s)} ds = \frac{\la}{2\pi} \int_A \frac{e^{u_{\la}(x)}}{|x|^2} dx.
\] 
Applying Proposition \ref{Prop:W(exp)} with $\bar{a} = -\log b$, $\bar{b} = -\log a$, we obtain the result.
\end{proof}

Similar argument using the transformation \eqref{transform(N>2)} and Proposition \ref{Prop:W(exp)} leads to the following. 
The simple proof of it is omitted.

\begin{theorem}
\label{Thm:(L)N>2}
There exists $\la_* > 0$ such that the problem
\begin{equation}
\label{Henon(L)(N>2)}
	\begin{cases}
	-\Delta u = \la (N-2)^2 \dfrac{e^{u(x)}}{|x|^{2(N-1)}}, \quad &x \in A = \{ a < |x| < b \} \subset \re^N, \quad (N \ge 3) \\
	u(x) = 0, \quad &x \in \pd A, 
	\end{cases}
\end{equation}
admits no radially symmetric solution for $\la > \la_*$,
one radially symmetric solution for $\la = \la_*$, 
and just two radially symmetric solutions $\ul{u}_{\la}$ and $U_{\la}$ for $0 < \la < \la_*$.
Also it holds that
\begin{align*}
	\la \int_A \frac{e^{\ul{u}_{\la}}}{|x|^{2(N-1)}} dx \to 0, \quad
	\la \int_A \frac{e^{U_{\la}}}{|x|^{2(N-1)}} dx \to +\infty,
\end{align*}
and $\ul{u}_{\la} \to 0$ uniformly on $\ol{A}$ as $\la \to +0$.

Define $\delta_{\la} > 0$ and $\tilde{U}_{\la}$ by
\begin{equation}
	\begin{cases}
	&\la \delta_{\la}^2 e^{\| U_{\la} \|_{L^{\infty}(A)}} \equiv 1, \\
	&\tilde{U}_{\la}(t) = U_{\la} \( (\delta_{\la} t + \frac{a^{2-N}+b^{2-N}}{2})^{\frac{1}{2-N}} \) - \| U_{\la} \|_{L^{\infty}(A)}, \\
	&t \in \tilde{I_{\la}} = (-\frac{a^{2-N} - b^{2-N}}{2\delta_{\la}}, \frac{a^{2-N} - b^{2-N}}{2\delta_{\la}}). 
	\end{cases}
\end{equation}
then $\delta_{\la} \to 0$ as $\la \to +0$ and
\[
	\tilde{U}_{\la}(t) \to U(t) = \log \frac{4 e^{\sqrt{2}t}}{(1+e^{\sqrt{2}t})^2} \quad \text{in} \ C^1_{loc}(\re)
\]
as $\la \to +0$.
Also
\[
	\delta_{\la} U_{\la}(|x|) \to 2\sqrt{2} G\(|x|^{2-N}, \frac{a^{2-N} + b^{2-N}}{2}\)
\] 
as $\la \to +0$ uniformly in $C^0(\ol{A})$,
where $G(s,t)$ is the Green's function of the operator $-\frac{d^2}{ds^2}$ with the Dirichlet boundary condition on $[\bar{a}, \bar{b}]$
defiend in \eqref{Green} with $\bar{a} = b^{2-N}$, $\bar{b} = a^{2-N}$.
\end{theorem}

\vspace{1em}
\noindent\textbf{Data Availability.} 
Data sharing is not applicable to this article as no datasets were generated or analysed during the current study.

\vspace{1em}
\noindent\textbf{Conflicts of Interest Statement.}
The author has no conflicts of interest to disclose.

\vspace{1em}
\noindent\textbf{Acknowledgement.} 
The author was supported by JSPS Grant-in-Aid for Scientific Research (B), No. 23H01084, 
and was partly supported by Osaka Central University Advanced Mathematical Institute (MEXT Joint Usage/Research Center on Mathematics and Theoretical Physics).

\end{document}